\documentclass[a4paper, 11pt]{paper}
\usepackage{geometry}

\usepackage[english]{babel}
\usepackage[utf8x]{inputenc}
\usepackage{amsmath}
\usepackage{amsthm}
\usepackage{amssymb}
\usepackage{dsfont}
\usepackage{graphicx}
\graphicspath{{pictures/}}
\usepackage[textsize=footnotesize]{todonotes}
\usepackage{listing}
\usepackage{xcolor}
\usepackage{textcomp}
\usepackage{comment}
\usepackage{hyperref}
\usepackage{wrapfig}
\usepackage{physics}
\usepackage{mathtools}
\usepackage{enumitem}
\usepackage{nameref}
\usepackage{thmtools}
\usepackage{pdfpages}
\usepackage{comment}

\usepackage{pgfplots}
\pgfplotsset{compat=1.12}
\usepackage{tikz}
\usepackage{tikzscale}
\usetikzlibrary{shapes,arrows,calc,math,3d}

\usepackage{subcaption}

\title{Limit theorems for the number of crossings and stress in projections of the random geometric graph}
\author{Hanna Döring
\thanks{Osnabrück University, Germany, hanna.doering@uni-osnabrueck.de} \and Lianne de Jonge\thanks{Osnabrück University, Germany, lianne.de.jonge@uni-osnabrueck.de}}
\date{\today}

\theoremstyle{definition}
\newtheorem{defn}{Definition}[section]
\newtheorem{prop}{Proposition}[section]
\newtheorem{rem}{Remark}[section]

\newtheorem{thm}{Theorem}[section]
\newtheorem{cor}{Corollary}[section]
\newtheorem{lem}{Lemma}[section]

\renewcommand{\l}{\left}
\renewcommand{\r}{\right}

\DeclareMathOperator{\Ex}{\mathbb{E}}
\DeclareMathOperator{\Var}{\mathbb{V}}

\DeclareMathOperator{\stress}{stress}

\DeclareMathOperator{\Cov}{\C ov}

\DeclareMathOperator{\Poi}{Poi}

\newcommand{\ind}{\mathds{1}}
\newcommand{\indi}[1]{\ind\l(#1\r)}

\newcommand{\C}{\mathbb{C}}

\newcommand{\E}{\mathbb{E}}

\newcommand{\R}{\mathbb{R}}

\newcommand{\cB}{\mathcal{B}}

\newcommand{\cH}{\mathcal{H}}

\newcommand{\cN}{\mathcal{N}}
\newcommand{\cO}{\mathcal{O}}

\newcommand{\bfB}{\mathbf{B}}

\newcommand{\bfL}{\mathbf{L}}
\newcommand{\bfM}{\mathbf{M}}

\renewcommand{\epsilon}{\varepsilon}
\renewcommand{\phi}{\varphi}

\newcommand{\radt}{r_t}

\begin{document}

\maketitle
\begin{abstract}
    We consider the number of edge crossings in a random graph drawing generated by projecting a random geometric graph on some compact convex set $W\subset \mathbb{R}^d$, $d\geq 3$, onto a plane. The positions of these crossings form the support of a point process. We show that if the expected number of crossings converges to a positive but finite value, this point process converges to a Poisson point process in the Kantorovich-Rubinstein distance. We further show a multivariate central limit theorem between the number of crossings and a second variable called the stress that holds when the expected vertex degree in the random geometric graph converges to a positive finite value.
\end{abstract}


\section{Introduction}
\label{sec:introduction}
We consider the random geometric graph (RGG) generated in a compact convex set $W\subset \R^d$ with $d\geq 3$ and volume one. The vertices of the graph are defined by a homogeneous Poisson point process on $W$ with intensity $t$. Two vertices $v$ and $w$ are connected by an (undirected) edge $\{v,w\}$ whenever $\|v-w\|\leq\radt$, where $\|\cdot\|$ denotes the Euclidean norm.

To visualize this graph, we embed its vertices into $\R^2$ and connect each pair with a line whenever an edge exists, in which case it is likely that some lines need to cross. A straight-line drawing with few edge crossings tends to be the most aesthetically pleasing, see \cite{PCJ96}.
The minimum number of crossings in such a drawing is called the rectilinear crossing number. Finding this number is a classical problem in graph theory where a lot is still unknown, see for example \cite{fox2019}, \cite{Pach2000} or the survey \cite{schaefer2013}.

One can naturally extend this problem to random graphs by studying the crossing number of a random graph or the number of crossings in a random drawing of a graph.
In \cite{Pach2000} and \cite{Spencer2002} the expectation of the crossing number of an Erd\H{o}s-R\'enyi random graph as well as concentration inequalities are derived, and a generalization to $k$-planarity is considered in \cite{Asplund2018}. An extension to weighted Erd\H{o}s-R\'enyi random graphs is studied in \cite{Mohar2011}, 
where the Bernoulli-weight for the presence of an edge is given by i.i.d random variables for each edge.
In \cite{arenasvelilla20231} and \cite{arenasvelilla2023} Arenas-Velilla et al. study the number of crossings of a fixed graph $G$ whose vertices are randomly placed in convex position. They prove a central limit theorem using Stein's method.

In this paper we study a setting introduced in \cite{Chimani2018}, 
a RGG where the graph drawing is created by projecting the vertices and edges onto a plane, as shown in Figure \ref{fig:projection}. Compared to the methods of drawing the graph used by \cite{arenasvelilla20231}, \cite{arenasvelilla2023}, and \cite{moon1965}, this method has the advantage that it preserves some of the geometric properties of the original graph.  We refer to the number of edge crossings as the \emph{crossing number of the projection}, or simply `crossing number' when its meaning is clear from context. By construction it serves as an upper bound for the rectilinear crossing number of the original RGG in $W$.

\begin{figure}
    \centering
    \includegraphics[width = .3\textwidth]{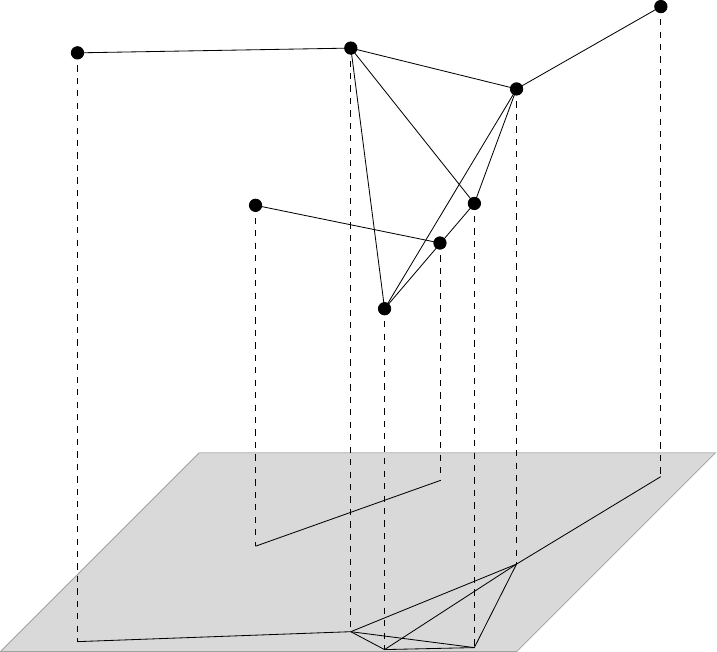}
    \caption{Projection of a graph in $\R^3$ onto a plane.}
    \label{fig:projection}
\end{figure}

The connection between the crossing number of a projection and another quantity called \emph{stress} was studied in \cite{Chimani2018}. The stress is a measure of how much the distances between vertices in the original graph and vertices in its drawing differ. In the context of projected random geometric graphs, it makes sense to consider the Euclidean distance between vertices in $\R^d$ and in the projection so that the graph stress is low when projecting barely changes the distances between vertices. However, if two vertices that are far away in $\R^d$ are close to each other in the projection, the stress is high. Empirical evidence suggests that drawings with low stress also tend to have fewer crossings. This is supported by \cite{Chimani2018}, where it is shown that there indeed exists a positive correlation between stress and the crossing number in the model described above.

This paper extends the results from \cite{Chimani2018} by proving a multivariate central limit theorem for the crossing number and stress. Additionally, we show that the point process of edge intersections converges to a Poisson point process when $\radt$ is chosen to decrease at an appropriate speed as $t$ tends to infinity.

It is easy to see that the expected degree of a randomly chosen vertex in the random geometric graph is of order $t\radt^d$. Based on this, we distinguish between the three regimes introduced in \cite{Penrose2003}:
\begin{enumerate}
    \item \emph{sparse regime}: $\lim_{t\to\infty}t\radt^d = 0$ implies that the expected degree tends to zero as $t$ goes to infinity;
    \item \emph{thermodynamic regime}: $\lim_{t\to\infty} t\radt^d = c \in (0,\infty)$ implies that the expected degree converges to some positive but finite constant;
    \item  \emph{dense regime}: $\lim_{t\to\infty} t\radt^d = \infty$ implies that the expected degree tends to infinity.
\end{enumerate}
The expected crossing number of the projection is of order $t^4\radt^{2d+2}$ (see \cite[Theorem 1]{Chimani2018}).
For the convergence to a Poisson point process we consider a part of the sparse regime where the expected crossing number of the projection tends to a positive but finite value, which is when $t^2\radt^{d+1} \to c \in(0,\infty)$ as $t$ tends to infinity. For the multivariate central limit theorem, we focus on the thermodynamic and the dense regime, for which the expected crossing number of the projection tends to infinity.

The study of the crossing number in random graphs is not only mathematically an interesting problem. As argued in \cite{Chimani2018}, the number of crossings in the projection in the sparse regime is indeed with high probability only a constant factor away from the rectilinear crossing number of the RGG. Therefore the projection provides an algorithm for drawing a graph with low crossing number. Limit theorems with rates of convergence are decisive for determining the approximation error in this algorithm.

The structure of this paper is as follows: Section \ref{sec:definitions notation} introduces the mathematical definitions and notation. Then, the convergence to a Poisson point process and the central limit theorem are stated, which are then proven in Section \ref{sec:sparse} and Section \ref{sec:multivariate clt} respectively.
The convergence to a Poisson point process is an application of a limit result for Poisson $U$-statistics proven by Decreusefond et al. using Stein's method and the Malliavin formalism \cite{Decreusefond2016}. The main difficulty in Section \ref{sec:sparse} is to derive uniform bounds on the intensity measure of the point process of crossings.
In Section \ref{sec:multivariate clt} we prove good bounds on the first and second order difference operators to apply a multivariate limit theorem for a vector of Poisson functionals by Schulte and Yukich \cite{Schulte2019}. 

\subsection{Definitions and statements of results}
\label{sec:definitions notation}
Let $\lambda_n$ for $n\geq 1$ denote the $n$-dimensional Lebesgue measure. Our observation window $W\subset \R^d$ with $d\geq 3$ is a convex body such that $\lambda_d(W) = 1$. Let $\eta_t$ be a Poisson point process on $W$ with intensity $t>0$. Then $\E \eta_t(W) = t$ since the volume of $W$ is one. The support of $\eta_t$ defines the vertex set $V$ of the \emph{random geometric graph} (RGG) $G = (V,E)$ with
\begin{equation*}
    E = \{\{v,w\} \colon\, v,w\in V \text{ and } \|v-w\|\leq \radt\}
\end{equation*}
for some parameter $\radt$ going to zero as $t$ tends to infinity. Here, $\|\cdot\|$ denotes the Euclidean norm. We denote the line segment between two points $v,w\in\R^d$ by $[v,w]$. Note that no points exist outside of $W$, meaning that none of the edges cross the boundary of $W$. Almost surely vertices are not projected on other vertices or edges.

The orthogonal projection of the graph onto a plane $L$ is denoted by $G_L$ and is constructed by projecting all vertices onto $L$ and connecting the vertices by an edge whenever an edge exists in $G$. The projection of a point $v$ and a set $A\subset \R^d$ are denoted by $v|_L$ and $A|_L\coloneqq \{v|_L\colon v\in A\}$ respectively. We write $W_L\coloneqq W|_L$ for the projection of $W$ onto $L$ to avoid cumbersome notation. The orthogonal complement of $L$ is denoted by $L^{\perp}$.

Two edges $\{v_1,v_2\}$ and $\{w_1,w_2\}$ between four distinct vertices cross in the projection if $[v_1,v_2]|_L\cap [w_1,w_2]|_L\neq \emptyset$. These intersections form the support of a point process $\xi_t$. Let
\begin{equation*}
    S_t \coloneqq \{(v_1,v_2,w_1,w_2)\in V^4_{\neq}\colon \|v_1-v_2\|\leq \radt,\, \|w_1-w_2\|\leq \radt,\, [v_1,v_2]|_L\cap [w_1,w_2]|_L\neq \emptyset\}
\end{equation*}
be the set of quadruples of distinct vertices forming two crossing edges. Then
\begin{equation}
    \xi_t(A) \coloneqq \frac{1}{8} \sum_{
    (v_1,v_2, w_1,w_2)\in S_t} 
    \delta_{[v_1,v_2]|_L\cap [w_1,w_2]|_L}(A)
    \label{eq:crossing point process}
\end{equation}
counts the number of crossings in a Borel set $A\subset L$. The pre-factor prevents the double counting of crossings. The crossing number of the projection is denoted $\xi_t(L)$. The intensity measure $\bfL_t$ of the point process is defined as
\begin{equation*}
    \bfL_t(A) \coloneqq \E\xi_t(A).
\end{equation*}

For all Borel sets $A$, the variable $\xi_t(A)$ is a $U$-statistic of order four, meaning that it can be written in the form
\begin{equation*}
    \xi_t(A) = \sum_{(v_1,v_2,v_3,v_4)\in V^4_{\neq}} f(v_1,v_2,v_3,v_4)
\end{equation*}
for some symmetric function $f$ depending on $A$. Here we have
\begin{align*}
    f(v_1,v_2,v_3,v_4) = &\frac{1}{24}(\indi{\|v_1-v_2\|\leq \radt,\, \|v_3-v_4\|\leq \radt,\,[v_1,v_2]|_L\cap [v_3,v_4]|_L\cap A\neq \emptyset}\\
    &+\indi{\|v_1-v_3\|\leq \radt,\, \|v_2-v_4\|\leq \radt,\,[v_1,v_3]|_L\cap [v_2,v_4]|_L\cap A\neq \emptyset}\\
    &+\indi{\|v_1-v_4\|\leq \radt,\, \|v_2-v_3\|\leq \radt,\,[v_1,v_4]|_L\cap [v_2,v_3]|_L\cap A\neq \emptyset}).
\end{align*}
Since we consider three permutations of the vertices in each summand, the pre-factor from \eqref{eq:crossing point process} is divided by three, explaining the pre-factor $\frac{1}{24}$.

Throughout this paper, we let $B_d(x,r) \coloneqq \{y\in\R^d\colon \|x-y\|\leq r\}$ denote the closed ball with radius $r$ around $x$ and $B_d \coloneqq B_d(0,1)$ the unit ball around the origin. The balls $B_2(x,r)$ and $B_2$ are always considered to lie in the plane $L$. The inner parallel set of a convex set $K$ is denoted by $K_{-\delta} \coloneqq \{x\colon\,(x+\delta B_d)\subset K\}$. We write $f(x) = \cO(g(x))$ as $x\to\infty$ if there exist $M>0$ and $x_0$ such that $|f(x)| \leq Mg(x)$ for all $x\geq x_0$. We also write $f(x) = \cO_x(g(x))$ and omit the $x\to\infty$.

\bigskip

Our first result concerns the distribution of crossings in the sparse regime. We show convergence of $\xi_t$ to a Poisson point process in the Kantorovich--Rubinstein distance. The definition of this distance metric involves the total variation distance. For two finite measures $\mu_1$ and $\mu_2$ on the Borel sets $\cB(L)$ of $L$, the total variation distance is defined as
\begin{equation*}
    \dd_{TV}(\mu_1, \mu_2) \coloneqq \sup_{A\in \cB(L)} | \mu_1(A) - \mu_2(A) |.
\end{equation*}
The Kantorovich--Rubinstein distance is then the optimal transport cost between two probability measures
\begin{equation*}
    \dd_{K\!R}(\mu_1, \mu_2) =\inf_{\mathbf C\in\Sigma(\mu_1,\mu_2)} \int \dd_{TV}(\omega_1, \omega_2) {\mathbf C}\bigl(\dd(\omega_1,\omega_2)\bigr)
\end{equation*}
where $\Sigma(\mu_1,\mu_2)$ is the set of all couplings between $\mu_1$ and $\mu_2$ and we integrate over all $\sigma$-finite counting measures on the underlying space. For two point processes $\xi_1$ and $\xi_2$ distributed according to $\mu_1$ and $\mu_2$ respectively, we will write $\dd_{K\!R}(\xi_1,\xi_2)$ instead of $\dd_{K\!R}(\mu_1,\mu_2)$.
By \cite[Proposition 2.1.]{Decreusefond2016}, convergence in $\dd_{K\!R}$ implies convergence in distribution.

The crossings in the RGG can also be studied in two dimensions without projection. In this case, the definitions above can easily be adjusted by removing all mentions of the projection onto $L$. The authors are not aware of results related to the distribution of these crossings and therefore include the $d=2$ case in the statement of the theorem below.
\begin{thm}[Convergence to Poisson point process]
    Consider the projected RGG with intensity $t$, dimension $d\geq 2$, and radius $\radt$ such that $t^2\radt^{d+1}\to c > 0$ as $t\to\infty$ and let the edge crossings form the support of the point process $\xi_t$, as defined in \eqref{eq:crossing point process}. Let $\zeta$ be a Poisson point process on $L$ with intensity measure $\bfM$ defined by
    \begin{equation*}
        \bfM(A) = \frac{1}{8} c_d c^2\int_A \lambda_{d-2}((v + L^{\perp})\cap W)^2\dd v,
    \end{equation*}
    for all Borel-sets $A\subset L$, where $c_d = 8\pi \kappa_{d-2}^2 \bfB\l(3,d/2\r)^2$ with $\bfB$ the beta function. Then,
    \begin{equation*}
        \dd_{K\!R}(\xi_t,\zeta) = \begin{cases}
            \cO_t(\sqrt{\radt}) + \cO_t(c^2 - t^4 \radt^{6}) & \text{if } d = 2\\
            \cO_t(\radt) + \cO_t(c^2 - t^4 \radt^{2d+2}) &\text{if } d \geq 3.
        \end{cases}
    \end{equation*}
    Convergence in distribution of $\xi_t$ to $\zeta$ follows.
    \label{thm:point process}
\end{thm}

The term $\cO_t(c^2 - t^4 \radt^{2d+2})$ indicates that the rate of convergence of $\xi_t$ to $\zeta$ depends on how fast $t^2\radt^{d+1}$ approaches its limiting value.

This theorem is proven in Section \ref{sec:sparse}. The proof relies on \cite[Theorem 3.1.]{Decreusefond2016}, which roughly states that a $U$-statistic $\xi_t$ converges to a Poisson point process if the intensity measure converges in total variation distance and the difference between the expectation and the variance of $\xi_t(L)$ goes to zero.

\begin{rem}
    The paper \cite{Chimani2018} also considered a random projection plane $L$, whereas we consider a fixed plane. In the sparse regime described above, the limiting distribution of the crossing number would be a mixed Poisson distribution where the intensity is a random measure depending on the orientation of $L$.
\end{rem}

\bigskip
\noindent
The second result concerns another quantity called stress. Let $\dd_0$ and $\dd_L$ be two distance metrics on the vertices in a graph, then
\begin{equation*}
    \stress(G,G_L) \coloneqq \frac{1}{2}\sum_{(v_1,v_2)\in V^2_{\neq}} w(v_1,v_2) \l(\dd_0(v_1,v_2) - \dd_L(v_1,v_2)\r)^2,
\end{equation*}
where $w$ is some weight function. We will denote the stress between two vertices by $$S(v_1,v_2;V) \coloneqq w(v_1,v_2) \l(\dd_0(v_1,v_2) - \dd_L(v_1,v_2)\r)^2.$$
The $V$ indicates that $S$ might depend on $V\setminus\{v_1,v_2\}$, which is the case if $\dd_0$ or $\dd_L$ depends on $V$. Without this dependence, the stress is a $U$-statistic, which is the assumption made in \cite{Chimani2018}.

A natural choice for $w$ is $$w(v_1, v_2) = \frac{1}{\dd_0(v_1,v_2)^2},$$ in which case
\begin{equation*}
    \stress(G,G_L) = \frac{1}{2}\sum_{(v_1,v_2)\in V^2_{\neq}}  \l(1 - \frac{\dd_L(v_1,v_2)}{\dd_0(v_1,v_2)}\r)^2.
\end{equation*}
For the RGG, it makes sense to let $\dd_0$ be the Euclidean distance between points in $\R^d$ and $\dd_L$ the Euclidean distance in the projection. Note that in this case $\dd_L \leq \dd_0$, from which $S(v_1,v_2;V)\in[0,1]$ follows for all $v_1,v_2\in W$. 

The expectation and variance of the crossing number and stress, as well as the covariance between the two quantities, have been calculated in \cite{Chimani2018} for the situation where the value of $S(v_1,v_2;V)$ only depends on the two vertices $v_1$ and $v_2$. An extension of their results to more general stress functions can be found in Appendix \ref{ap:variance}.

The known expressions for the variances and covariance can be used to transform the crossing number of the projection, $\xi_t(L)$, and the stress to have mean zero and bounded variance. 
In Section \ref{sec:multivariate clt}, we use these results and \cite[Theorem 1.1]{Schulte2019} to show a central limit theorem in the $\dd_3$-distance, which is defined as follows:
\begin{defn}
    Let $\cH$ be the set of functions $g\colon \R^d\to\R$ admitting continuous partial derivatives up to order three such that
    \begin{align*}
        &\max_{1\leq i_1 \leq i_2 \leq d} \sup_{x\in\R^d}\l|\frac{\partial^2}{\partial x_{i_1}\partial x_{i_2}}g(x)\r|\leq 1 \text{, and } &\max_{1\leq i_1 \leq i_2 \leq i_3\leq d} \sup_{x\in\R^d}\l|\frac{\partial^3}{\partial x_{i_1}\partial x_{i_2} \partial x_{i_3}}g(x)\r|\leq 1
    \end{align*}
    Then for two random variables $X$ and $Y$ such that $\E\|X\|^2<\infty$ and $\E\|Y\|^2<\infty$ we define
    \begin{equation*}
        \dd_3(X,Y) \coloneqq \sup_{g\in\cH} |\E g(Y) - \E g(X)|.
    \end{equation*}
\end{defn}

\begin{thm}[Central limit theorem]
    Let $d\geq 3$ and consider the projected RGG in the thermodynamic regime where $t\radt^d\to c>0$ or the dense regime where $t\radt^d\to\infty$ as $t$ tends to infinity. Let $S(v_1,v_2; V) \leq s$ almost surely for all $(v_1,v_2)\in V^2_{\neq}$ and define
    \begin{equation*}
        F_t = (F_t^{(1)}, F_t^{(2)}) \coloneqq \l(\frac{\xi_t(L) - \E\xi_t(L)}{t^{7/2}\radt^{2d+2}}, \frac{\stress(G,G_L) - \E\stress(G,G_L)}{t^{3/2}}\r).
    \end{equation*}
    Let $\Sigma_t$ be the covariance matrix of $F_t$. Then
    \begin{equation*}
        \dd_3(F_t, Z_{\Sigma_t}) = \cO_t(t^{-1/2}),
    \end{equation*}
    where $Z_{\Sigma_t}\sim \cN(0,\Sigma_t)$ is a multivariate normal distributed random variable with mean zero and covariance matrix $\Sigma_t$.
    Additionally, if there exists a $2\times 2$ matrix $\Sigma = \lim_{t\to\infty} \Sigma_t$ then $\dd_3(F_t,Z_{\Sigma})\to 0$ as $t\to\infty$, where $Z_\Sigma\sim\cN(0,\Sigma)$. This implies convergence in distribution of $F_t$ to $Z_\Sigma$.
    \label{thm:multivariate CLT}
\end{thm}

Conditions on the stress such that the covariance matrix converges are derived in Appendix \ref{ap:variance}.
The rate of convergence to $Z_\Sigma\sim \cN(0,\Sigma)$ depends on how fast the (co)variances converge to their limiting value, which we were not able to derive. In Appendix \ref{sec:cube}, we derive $\dd_3(F_t,Z_\Sigma) = \cO_t(\radt)$ for a special case with $W=[0,1]^d$ and $L=\R^2\times\{0\}^{d-2}$, which is mostly the result of boundary effects. We expect this to be a typical rate of convergence since boundary effects are present in all models.

\begin{rem}[Stress as a $U$-statistic]
    Note that if the stress is a $U$-statistic with $S(v_1,v_2;V) = S(v_1,v_2)<s$ (for example if $\dd_0$ and $\dd_L$ are Euclidean metrics), a uni-variate central limit theorem for the stress also follows immediately from \cite[Theorem 3]{Last2014}, in which case we have
    \begin{equation*}
        \dd_3\l(\frac{\stress(G,G_L) - \E\stress(G,G_L)}{\sqrt{\Var\stress(G,G_L)}}, Z\r) = \cO_t(t^{-1/2}),
    \end{equation*}
    where $Z\sim\cN(0,1)$. By \cite[Corollary 4.3]{schulte_normal_2016}, this convergence also holds in the Kolmogorov distance.

    The result \cite[Theorem 1.1]{Schulte2019} that is applied to prove the central limit theorem can also be used to show convergence in the $\dd_2$-distance if the limiting covariance matrix is positive-definite. For a $2\times 2$ covariance matrix, it is enough to show that $$\bigl(\lim_{t\to\infty}\Cov(F_t^{(1)},F_t^{(2)})\bigr)^2<\lim_{t\to\infty}\mathbb{V} F_t^{(1)} \mathbb{V} F_t^{(2)}.$$
    If $S(v_1,v_2;V) = S(v_1,v_2)<s$, this  inequality quickly follows from an application of the Cauchy-Schwarz inequality in combination with the variance and covariance expressions from \cite{chimani2020Arxiv}. We were not able to show positive-definiteness for general stress functionals.
\end{rem}

The Slivnyak-Mecke formula (see for example \cite{Last_Penrose_2017}) is a well-known formula that is particularly useful in the proofs of both theorems. In the context of the notation used in this paper, this formula can be written as
\begin{equation}
    \E\Big[\sum_{v_1,\dots,v_n\in V^n_{\neq}} f(v_1,\dots,v_n,V)\Big] = t^n\int_{W^n} \E\l[f(v_1,\dots,v_n,V\cup\{v_1,\dots,v_n\})\r]\dd v_1\cdots \dd v_n
    \label{eq:mecke}
\end{equation}
for all measurable functions $f$.

This paper builds on the work done in \cite{Chimani2018}. However, the proofs and derivations from \cite{Chimani2018} are explained more elaborately in the arXiv paper \cite{chimani2020Arxiv}. When results or proofs from the arXiv paper are needed, we refer to that paper instead.

\section{Convergence of crossings to a Poisson point process}
\label{sec:sparse}
In this section, we prove Theorem \ref{thm:point process} by applying \cite[Theorem 3.1]{Decreusefond2016}.
In order to apply this theorem, we need to verify two things: the intensity measure needs to converge to some finite measure and the difference between the expectation and the variance must converge to zero. We first formulate how $\xi_t$ satisfies these conditions and then prove Theorem \ref{thm:point process}. The Lemmas are then proven in Section \ref{sec:proofs lemmas}.

Our first result is on the intensity measure of the crossings in $L$. Compared to the expected crossing number $\bfL_t(L)$ derived by \cite{Chimani2018}, this lemma features an improvement on the error term, here denoted by $g_t(A)$.
\begin{lem}[Intensity measure]
    There exists a term $g_t(A)$ depending on $t$ and $A$ such that
    \begin{align*}
        \bfL_t(A) = \frac{1}{8} c_d t^4\radt^{2d+2} \Bigl(\int_{A} \lambda_{d-2}((v + L^{\perp})\cap W)^2\dd v + g_t(A) \Bigr)
    \end{align*}
    for all Borel-sets $A\subset L$, with $c_d = 8\pi \kappa_{d-2}^2 \bfB\l(3,d/2\r)^2$. The term $g_t$ is uniformly bounded for all $A$, i.e. $|g_t(A)| \leq c_t = \cO_t(\radt)$,
    where $c_t$ is a constant depending on the intensity $t$ and the choice of $W$ and $L$.
    \label{lem:intensity measure}
\end{lem}
The following bound follows immediately:

\begin{cor}[Total variation distance]
    Let
    $$\bfM(A) = \frac{1}{8} c_d c^2\int_A \lambda_{d-2}((v + L^{\perp})\cap W)^2\dd v$$
    be the measure defined as in Theorem \ref{thm:point process}
    and $t^2\radt^{d+1}\to c$ as $t\to\infty$.
    Then
    \begin{equation*}
        \dd_{TV}(\bfL_t,\bfM) = \sup_{A\in\cB(L)}|\bfL_t(A) - \bfM(A)| = \cO_t(\radt) +\cO_t(c^2 - t^4\radt^{2d+2})
    \end{equation*}
    goes to zero as $t$ tends to infinity.
    \label{lem:tv distance}
\end{cor}

\begin{lem}[Convergence variance to expectation]
    Let $t^2\radt^{d+1}\to c$ as $t\to\infty$. Then,
    \begin{equation*}
        \Var \xi_t(L) - \Ex \xi_t(L) = \begin{cases}
            \cO_t(\sqrt{\radt}) &\text{if } d=2\\
            \cO_t(\radt) &\text{if } d\geq 3
        \end{cases}
    \end{equation*}
    as $t$ tends to infinity.
    \label{lem:convergence var exp}
\end{lem}

The proof of Theorem \ref{thm:point process} now quickly follows.
\begin{proof}[Proof of Theorem \ref{thm:point process}]
    Let $\zeta$ be a Poisson point process with finite intensity measure $\bfM$. From \cite[Theorem 3.1.]{Decreusefond2016}, we know that
    \begin{align*}
        \dd_{K\!R}(\xi_t,\zeta) \leq \dd_{TV}(\mathbf{L}_t, \mathbf{M}) + 2(\E\xi_t(L)^2 - \E\xi_t(L) - (\E\xi_t(L))^2).
    \end{align*}
    Combining this with Corollary \ref{lem:tv distance} and Lemma \ref{lem:convergence var exp}, we obtain the desired result.
\end{proof}

\subsection{Proofs of the Lemmas}
\label{sec:proofs lemmas}
In this section, we prove Lemmas \ref{lem:intensity measure} and \ref{lem:convergence var exp}. Although our intensity measure looks similar to the expected crossing number derived in \cite{Chimani2018}, we use a different approach to derive it for any Borel set $A$. To count all crossings in $L$, it is sufficient to count any crossing resulting from four vertices in $W$, which was used by \cite{Chimani2018} to calculate $\E\xi_t(L)$. Since vertices projected outside $A$ can have crossing edges inside $A$, we need to be more careful in the derivation.

The Poisson point process is homogeneous in $W$ but the projected vertices are generally not homogeneous in $W_L$, which further complicates the proof. In the proof of the lemma below, the expected number of crossings in a sufficiently small Borel set $U\subset L$ is bounded using lower and upper bounds on the intensity of projected vertices near $U$.

\begin{lem}[Bounds on the intensity measure]
    The intensity measure of $\xi_t$ is bounded as follows:
    \begin{align*}
        \frac{1}{8} c_d t^4\radt^{2d+2}\int_{A}\inf_{w\in B_2(v, 2\radt)} \lambda_{d-2}((w + L^{\perp})\cap W_{-\radt})^2\dd v \leq \bfL_t(A) \\
        \leq \frac{1}{8} c_d t^4\radt^{2d+2}\int_{A}\sup_{w\in B_2(v, 2\radt)} \lambda_{d-2}((w + L^{\perp})\cap W)^2\dd v
    \end{align*}
    for all Borel set $A\subset L$.
    \label{lem:intensity measure bounds}
\end{lem}

\begin{proof}
    Let $A\subset L$ be a Borel set. Then,
    \begin{align*}
        \bfL_t(A) &=\E\Big[\frac{1}{8}\sum_{(v_1,...,v_4)\in V^4_{\neq}}
        \begin{multlined}[t]\indi{[v_1,v_3]|_L\cap [v_2,v_4]|_L \cap A \neq \emptyset}\\
            \cdot\indi{\|v_1-v_3\|\leq\radt,\, \|v_2-v_4\|\leq \radt}\Big]
        \end{multlined}
        \\
        &=\frac{1}{8}t^4\int_{W^4}
        \begin{multlined}[t]
            \indi{[v_1,v_3]|_L\cap [v_2,v_4]|_L \cap A \neq \emptyset}\\
            \cdot\indi{\|v_1-v_3\|\leq\radt,\, \|v_2-v_4\|\leq \radt}\dd v_1\cdots\dd v_4,
        \end{multlined}
    \end{align*}
    where the second equality follows from the Slivnyak-Mecke formula \eqref{eq:mecke}.

    For $x\in A$, let $U_x\subset B_2(x,\radt)$ be a non-empty Borel subset of $L$. The set $A$ can be written as the union of such disjoint sets $U_x$. In the proofs of the lower and upper bounds below, we derive a bound on the intensity of crossings in a set $U_x$, $\bfL_t(U_x)$. These bounds can be used to define a sequence of simple functions converging almost everywhere to the integrands from the lemma. The bounds for $\bfL_t(A)$ then follow from an application of the dominated convergence theorem.

    \paragraph{Lower bound}
    The position of a vertex on one end of an edge can be described by its position relative to the other vertex, leading to the following lower bound:
    \begin{multline*}
        \bfL_t(U_x)\geq
        \frac{1}{8}t^4\int_{(W_{-\radt})^2}\int_{(\radt B_d)^2}
            \indi{(v_1 + [0,w_1])|_L\cap (v_2 + [0,w_2])|_L \cap U_x \neq \emptyset}\\
            \dd w_1 \dd w_2 \dd v_1 \dd v_2
        \eqqcolon \text{RHS},
    \end{multline*}
    where $v_1 + [0,w_1] = [v_1, v_1 + w_1]$ is a translation of the line $[0,w_1]$ by $v_1$.
    By restricting $v_1$ and $v_2$ to lie in the inner parallel set $W_{-\radt}$, we ensure that the points $v_1 + w_1$ and $v_2 + w_2$ lie in $W$.
    \begin{figure}
        \centering
        \begin{subfigure}{.45\textwidth}
            \centering
            \includegraphics[width = .6\linewidth]{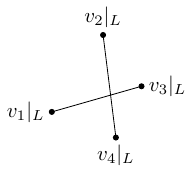}
            \caption{Construction of a crossing in $L$ with 4 vertices}
        \end{subfigure}
        \hfill
        \begin{subfigure}{.45\textwidth}
            \centering
            \includegraphics[width = .8\linewidth]{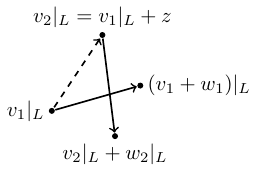}
            \caption{Construction of vertex positions in $L$ relative to $v_1|_L$}
            \label{fig:vertex construction vector}
        \end{subfigure}
        \caption{}
    \end{figure}
    This situation and another transformation coming up later in the proof are depicted in Figure \ref{fig:vertex construction vector}.

    An integral over a subset of $\R^d$ can be split into an integral over $L$ and one over $L^\perp$. Let
    \begin{equation*}
        \mu_{-\radt }(B) = \int_B \lambda_{d-2}((x+L^{\perp})\cap W_{-\radt }) \dd x
    \end{equation*}
    be the volume of $W_{-\radt }$ that gets projected onto a Borel set $B\subset L$. Then, the right-hand side of the inequality above can be expressed as
    \begin{multline*}
        \text{RHS} =  \frac{1}{8}t^4\int_{W_L^2}\int_{(\radt B_d)^2}\indi{(v_1 + [0, w_1]|_L)\cap (v_2 + [0, w_2]|_L) \cap U_x \neq \emptyset}\\
        \dd w_1 \dd w_2 \mu_{-\radt }(\dd v_2) \mu_{-\radt }(\dd v_1).
    \end{multline*}
    Note that the vertices $v_1$ and $v_2$ in this last expression lie in the plane $L$ instead of in $W$.

    If $(v_1 + [0, w_1]|_L)\cap (v_2 + [0, w_2]|_L) \cap U_x \neq \emptyset$ then $\|v_1 - x\| \leq 2 \radt $ and $\|v_2 - x\| \leq 2 \radt$. That is, for the indicator function in our integral to be non-zero, the points $v_1$ and $v_2$ must lie in the set $B_2(x,2 \radt )$. We define for all Borel sets $B\subset L$ the translation invariant measure
    \begin{equation*}
        \mu_-(B) \coloneqq \big(\inf_{w\in B_2(x,2\radt)} \lambda_{d-2}((w + L^{\perp})\cap W_{-\radt})\big)\lambda_2(B).
    \end{equation*}
    For all $B\subset B_2(x, 2\radt)$, we have $\mu_-(B) \leq \mu_{-\radt}(B)$, which implies
    \begin{multline*}
        \bfL_t(U_x) \geq \text{RHS}
        \geq\frac{1}{8}t^4 \int_{W_L^2}\int_{(\radt B_d)^2}
        \indi{(v_1 + [0,w_1]|_L)\cap (v_2 + [0, w_2]|_L) \cap U_x \neq \emptyset}\\
        \dd w_1 \dd w_2 \mu_-(\dd v_2) \mu_-(\dd v_1).
    \end{multline*}

    If the indicator function is non-zero then the position of $v_2\in L$ can be described relative to $v_1\in L$ using a vector $z\in 2\radt B_2$ given by $z = v_2 - v_1 $, see Figure \ref{fig:vertex construction vector}. Then,
    \begin{multline*}
        \bfL_t(U_x)
        \geq \frac{1}{8}t^4\int_{W_L}\int_{2\radt B_2}\int_{(\radt B_d)^2}
        \indi{(v_1 + [0,w_1]|_L)\cap ((v_1 + z) + [0,w_2]|_L) \cap U_x \neq \emptyset}\\
        \dd w_1 \dd w_2 \mu_-(\dd z) \mu_-(\dd v_1).
    \end{multline*}
    Whether or not $(v_1 + [0,w_1]|_L)$ and $(v_1 + z + [0,w_2]|_L)$ intersect is independent of the position of $v_1$. Combining this with an application of Fubini's theorem, we can write
    \begin{multline*}
        \bfL_t(U_x)
        \geq \frac{1}{8}t^4\int_{L}\int_{(\radt B_d)^2}
        \indi{[0,w_1]|_L\cap (z+[0,w_2]|_L)\neq \emptyset}\\
        \int_{W_L}\indi{(v_1 + [0,w_1]|_L)\cap ((v_1 + z) + [0,w_2]|_L) \cap U_x \neq \emptyset}\mu_-(\dd v_1)
        \dd w_1 \dd w_2 \mu_-(\dd z) .
    \end{multline*}
    
    Whenever $[0,w_1]|_L\cap (z+[0, w_2]|_L)\neq\emptyset$, we obtain the following equality by translation invariance of $\mu_-$:
    \begin{multline*}
        \int_{W_L}\indi{(v_1 + [0,w_1]|_L)\cap ((v_1 + z) + [0,w_2]|_L) \cap U_x \neq \emptyset}\mu_-(\dd v_1) =\mu_-(U_x)\\
        = \big(\inf_{w\in B_2(x,2\radt)} \lambda_{d-2}((w + L^{\perp})\cap W_{-\radt})\big)\lambda_2(U_x).
    \end{multline*}
    That is, if the position of the crossing relative to $v_1\in L$ is fixed, then $v_1$ can lie in a set of the same size as $U_x$. Note that if $x$ is close to the boundary of $W_L$ then $\mu_-(U_x) = 0$.
    Plugging this into the expression above, we obtain
    \begin{align*}
        \bfL_t(U_x)
        &\geq \begin{multlined}[t]\frac{1}{8}t^4\big(\inf_{w\in B_2(x,2\radt)} \lambda_{d-2}((w + L^{\perp})\cap W_{-\radt})\big)\lambda_2(U_x)\\\int_{2\radt B_2}
        \int_{(\radt B_d)^2}\indi{[0,w_1]|_L\cap (z + [0, w_2]|_L)\neq\emptyset}\dd w_1 \dd w_2 \mu_-(\dd z)
        \end{multlined}\\
        &= \begin{multlined}[t]\frac{1}{8}t^4\radt^{2d+2}\big(\inf_{w\in B_2(x,2\radt)} \lambda_{d-2}((w + L^{\perp})\cap W_{-\radt})\big)^2\lambda_2(U_x)\\
        \int_{2 B_2}\int_{(B_d)^2}\indi{[0,w_1]|_L\cap (z+[0, w_2]|_L)\neq\emptyset}\dd w_1 \dd w_2 \dd z
        \end{multlined}\\
        &= \frac{1}{8}t^4\radt^{2d+2}c_d\big(\inf_{w\in B_2(x,2\radt)} \lambda_{d-2}((w + L^{\perp})\cap W_{-\radt})\big)^2\lambda_2(U_x),
    \end{align*}
    where the calculation of $$c_d\coloneqq \int_{2 B_2}\int_{(B_d)^2}\indi{[0,w_1]|_L\cap (z+[0, w_2]|_L)\neq\emptyset}\dd w_1 \dd w_2 \dd z = 8\pi \kappa_{d-2}^2 \bfB\l(3,\frac{d}{2}\r)^2$$ is done in \cite{chimani2020Arxiv}.

    \paragraph{Upper bound} The proof of the upper bound is similar to the one of the lower bound, but for a supremum instead of an infimum. Boundary effects can be ignored since counting more edges is allowed for the upper bound. Then,
    \begin{multline*}
        \bfL_t(U_x)\leq \frac{1}{8}t^4\int_{W^2}\int_{(\radt B_d)^2}\indi{(v_1 + [0,w_1])|_L\cap (v_2 + [0,w_2])|_L \cap U_x \neq \emptyset}
        \\
        \dd w_1 \dd w_2 \dd v_1 \dd v_2.
    \end{multline*}
    Defining for all Borel sets $B\subset L$ the measure
    \begin{align*}
        \mu(B) \coloneqq \int_B \lambda_{d-2}((x+L^{\perp})\cap W)\dd x,
    \end{align*}
    we can transform our integrals to ones over the plane $L$:
    \begin{multline*}
        \bfL_t(U_x) \leq \frac{1}{8}t^4\int_{W_L^2}\int_{(\radt B_d)^2}\\
        \indi{(v_1 + [0,w_1])|_L\cap (v_2 + [0,w_2])|_L \cap U_x \neq \emptyset}\dd w_1 \dd w_2 \mu(\dd v_2) \mu(\dd v_1).
    \end{multline*}

    Let $z=v_2-v_1$ be the position of $v_2\in L$ relative to $v_1\in L$. Then,
    \begin{multline*}
        \bfL_t(U_x)\leq \frac{1}{8}t^4\int_{W_L}\int_{2\radt B_2}\int_{(\radt B_d)^2}\\
        \indi{(v_1 + [0,w_1])|_L\cap (v_1 + z + [0,w_2])|_L \cap U_x \neq \emptyset}\dd w_1 \dd w_2 \mu(\dd z) \mu(\dd v_1).
    \end{multline*}
    
    Let
    \begin{align*}
        \mu_+(B) \coloneqq \big(\sup_{w\in B_2(x,2\radt)} \lambda_{d-2}((w + L^{\perp})\cap W)\big)\lambda_2(B).
    \end{align*}
    For all $B\subset B_2(x,2\radt)$ we have $\mu_+(B) \geq \mu(B)$
    and therefore
    \begin{multline*}
        \begin{multlined}[t]
        \bfL_t(U_x) \leq
            \frac{1}{8}t^4\int_{W_L}\int_{B_2(v_1, 2\radt)}\int_{(\radt B_d)^2}\\
            \indi{(v_1 + [0,w_1])|_L\cap (v_1+z + [0,w_2])|_L \cap U_x \neq \emptyset}\dd w_1 \dd w_2 \mu_+(\dd z) \mu_+(\dd v_1)
        \end{multlined}
        \\
        =\begin{multlined}[t]\frac{1}{8}t^4\int_{2\radt B_2}\int_{(\radt B_d)^2}\indi{[0,w_1]|_L\cap (z + [0, w_2])|_L\neq\emptyset}\\
        \int_{W_L}\indi{(v_1 + [0,w_1])|_L\cap (v_1 + z + [0,w_2])|_L \cap U_x \neq \emptyset}\mu_+(\dd v_1)\dd w_1 \dd w_2 \mu_+(\dd z).
        \end{multlined}
    \end{multline*}
    Whenever $[0,w_1]|_L\cap (z + [0, w_2])|_L\neq\emptyset$, the translation invariance of $\mu_+$ implies
    \begin{multline*}
        \int_{W_L}\indi{(v_1 + [0,w_1])|_L\cap (v_1 + z + [0,w_2])|_L \cap U_x \neq \emptyset}\mu_+(\dd v_1)=\mu_+(U_x)\\
        = \big(\sup_{w\in B_2(x,2\radt)} \lambda_{d-2}((w + L^{\perp})\cap W)\big)\lambda_2(U_x).
    \end{multline*}
    Plugging this into the expression above, we obtain
    \begin{align*}
        \bfL_t(U_x)
        &\leq \frac{1}{8}t^4\radt^{2d+2}c_d\big(\sup_{w\in B_2(x,2\radt)} \lambda_{d-2}((w + L^{\perp})\cap W)\big)^2\lambda_2(U_x).
    \end{align*}
    The bound for $\bfL_t(A)$ with a Borel set $A\subset L$ follows from an application of the dominated convergence theorem as described at the start of this proof.
\end{proof}

The value of $\bfL_t(A)$ can be approximated using the lower and upper bounds from Lemma \ref{lem:intensity measure bounds}. Using the difference between these bounds, we can determine the maximum error of this estimation.
\begin{lem}[Uniform error bound]
    Let
    \begin{equation*}
        f_t(A) \coloneqq \int_{A}\sup_{w\in B_2(v, 2\radt)} \lambda_{d-2}((w + L^{\perp})\cap W)^2 - \inf_{w\in B_2(v, 2\radt)} \lambda_{d-2}((w + L^{\perp})\cap W_{-\radt})^2\dd v.
    \end{equation*}
    Then,
    \begin{align*}
        0 \leq f_t(A) \leq f_t(L) = \cO_t(\radt)
    \end{align*}
    for all Borel sets $A\subset L$.
    \label{lem:uniform error bound}
\end{lem}
\begin{proof}
    The integrand is non-negative for all points $v\in L$. Since $A\subset L$, it follows that $f_t(A)\leq f_t(L)$.
    
To show the order of the upper bound, we start by bounding $\lambda_{d-2}((w + L^{\perp})\cap W_{-\radt})^2$ to get rid of the $-\radt$. Note that
\begin{align*}
    \lambda_{d-2}((w + L^{\perp})\cap W_{-\radt})^2 &=
    \l(\lambda_{d-2}((w + L^{\perp})\cap W) - \lambda_{d-2}((w + L^{\perp})\cap (W\setminus W_{-\radt}))\r)^2\\
    &\geq \begin{multlined}[t]
    \lambda_{d-2}((w + L^{\perp})\cap W)^2 \\- 2 \lambda_{d-2}((w + L^{\perp})\cap W)  \cdot \lambda_{d-2}((w + L^{\perp})\cap (W\setminus W_{-\radt}))
    \end{multlined}
\end{align*}

For a set $A$, let $\partial A$ denote its boundary. Then, the volume of $W\setminus W_{-\radt}$ contained above $w$ can be bounded as follows:
\begin{align*}
    \lambda_{d-2}((w + L^{\perp})\cap (W\setminus W_{-\radt})) &\leq \radt \lambda_{d-3}(\partial((w+L^\perp)\cap W))\\
    &\leq \radt \sup_{v\in W_L}\lambda_{d-3}(\partial((v+L^\perp)\cap W)).
\end{align*}
Defining the constant
\begin{align*}
    b \coloneqq 2
    \big(\sup_{v\in W_L}\lambda_{d-3}(\partial((v+L^\perp)\cap W))\big)
    \big(\sup_{v\in W_L} \lambda_{d-2}((v + L^{\perp})\cap W)\big),
\end{align*}
we see that
\begin{align*}
    \lambda_{d-2}((w + L^{\perp})\cap W_{-\radt})^2 &\geq \lambda_{d-2}((w + L^{\perp})\cap W)^2 - b\radt
\end{align*}
for all $w\in W_L$.

    For $v\in L$, let
    \begin{align*}
        h(v) &\coloneqq \lambda_{d-2}((v + L^{\perp})\cap W)^2,\\
        h_+(v) &\coloneqq \sup_{w\in B_2(v, 2\radt)} \lambda_{d-2}((w + L^{\perp})\cap W)^2,\\
        h_-(v) &\coloneqq \inf_{w\in B_2(v, 2\radt)} \lambda_{d-2}((w + L^{\perp})\cap W)^2.
    \end{align*}
    Note that $h_+$ is supported on a subset of $L$ containing $W_L$. Hence,
    \begin{align}
        f_t(L) &\leq \int_{W_L} h_+(v) - h_-(v) + b\radt\, \dd v + \int_{L\setminus W_L} h_+(v)\, \dd v\\
        &= b\radt\lambda_2(W_L) + \int_{L} h_+(v) - h(v)\, \dd v + \int_{L} h(v) - h_-(v)\, \dd v.
        \label{eq:two remaining integrals}
    \end{align}
    These integrals represent the differences between the volumes under the surface defined by $h$, $h_+$, and $h_-$.
    For $t > 0$, let
    \begin{align*}
        C(z) &= \{v\in L \colon \, h(v)\geq z\},\\
        C_+(z) &= \{v\in L \colon \, h_+(v)\geq z\}
    \end{align*}
    be the superlevel sets of $h$ and $h_+$ respectively. If $h_+(v) = z$ for some $v\in L$ and $z\in\R$ then there must exist a point $w\in B_2(v, 2\radt)$ such that $h(w) = z$ too. Combining this with $h(v) \leq h_+(v)$ for all $v\in L$, we obtain $C(z) \subset C_+(z) \subset C(z) \oplus B_2(v,2\radt)$. By convexity of $W$, the function $\lambda_{d-2}((v + L^{\perp})\cap W)^{\frac{1}{d-2}} = h(v)^{\frac{1}{2(d-2)}}$ is concave, which implies that $C(z)$ is convex for all $z\in \R_+$. We can therefore apply the Steiner formula to see that $\lambda_2(C_+(z)\setminus C(z)) = \cO(\radt)$ for all $z$. Then,
    \begin{equation*}
        \int_{L} h_+(v) - h(v)\, \dd v = \int_0^\infty \lambda_2(C_+(z)\setminus C(z))\dd z = \cO_t(\radt)
    \end{equation*}
    is the order of the first integral in \eqref{eq:two remaining integrals}. Similar calculations show that $\int_L h(v) - h_-(v)\, \dd v = \cO_t(\radt)$ too, from which we conclude that $f_t(L) = \cO_t(\radt)$.

\begin{figure}
    \centering
    \includegraphics[width = .6\textwidth]{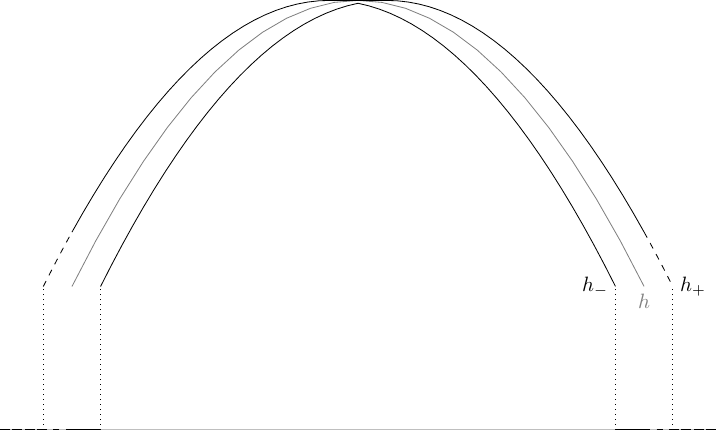}
    \caption{Two-dimensional representation of a function $h(v)$ (gray), the infimum $h_-(v)$
    and the supremum $h_+(v)$.
    Solid lines indicate the domain $W_L$ and dashed lines the domain $L\setminus W_L$.}
    \label{fig:h clarification}
\end{figure}
    
\end{proof}

We are now ready to prove Lemma \ref{lem:intensity measure}.
\begin{proof}[{\bf Proof of Lemma \ref{lem:intensity measure}}]
    Note that
    \begin{multline*}
        \inf_{w\in B_2(v, 2\radt)} \lambda_{d-2}((w + L^{\perp})\cap W_{-\radt})^2 \leq \lambda_{d-2}((v + L^{\perp})\cap W)^2\\
        \leq \sup_{w\in B_2(v, 2\radt)} \lambda_{d-2}((w + L^{\perp})\cap W)^2
    \end{multline*}
    for all $v\in L$.
    Combining  this with Lemma \ref{lem:intensity measure bounds}, the following inequalities follow:
    \begin{multline*}
        \frac{1}{8} c_d t^4\radt^{2d+2}\l(\int_{A} \lambda_{d-2}((v + L^{\perp})\cap W)^2\dd v - f_t(A)\r) \leq \bfL_t(A) \\
        \leq \frac{1}{8} c_d t^4\radt^{2d+2}\l(\int_{A} \lambda_{d-2}((v + L^{\perp})\cap W)^2\dd v + f_t(A)\r)
    \end{multline*}
    for all Borel sets $A\subset L$. We can therefore define $g_t$ to be the function such that
    \begin{align*}
        \bfL_t(A) = \frac{1}{8} c_d t^4\radt^{2d+2} \Bigl(\int_{A} \lambda_{d-2}((w + L^{\perp})\cap W)^2\dd v  + g_t(A) \Bigr).
    \end{align*}
    Combining this with Lemma \ref{lem:uniform error bound}, $|g_t(A)| \leq f_t(A) \leq f_t(L) = \cO_t(\radt)$.
    Defining $c_t \coloneqq f_t(L)$, the result follows.
\end{proof}

Lastly, we prove that the difference between the variance and the expectation converges to zero as $t$ tends to infinity.
\begin{proof}[{\bf Proof of Lemma \ref{lem:convergence var exp}}]
    For this proof, we can mostly repeat the arguments of \cite[Theorem 13]{chimani2020Arxiv}, which we will shorten here. Writing
    \begin{equation*}
        f(v_1,\dots,v_4) = \indi{[v_1,v_3]|_L\cap [v_2,v_4]|_L \neq \emptyset}\indi{\|v_1-v_3\|\leq\radt,\, \|v_2-v_4\|\leq \radt},
    \end{equation*}
    the variance of the crossing number is given by
    \begin{equation*}
        \Var \xi_t(L) = \frac{1}{64} \E \sum_{\substack{(v_1,\dots,v_4)\in V^4_{\neq}\\(w_1,\dots,w_4)\in V^4_{\neq}\\ |\{v_1,\dots,v_4\}\cap \{w_1,\dots,w_4\}|\geq 1}} f(v_1,\dots,v_4)f(w_1,\dots,w_4).
    \end{equation*}
    The sum over all quadruples such that $\{v_1,\dots,v_4\}=\{w_1,\dots,w_4\}$ yields the expected crossing number. The other terms are also calculated by \cite{chimani2020Arxiv} and by the last line of their proof, we have
    \begin{align*}
        \Var\xi_t(L) = \cO_t (t^7 \radt^{4d+4}) + \cO_t (t^6 \radt^{3d+4}) + \cO_t (t^6 \radt^{4d+2}) +  \cO_t (t^5 \radt^{3d+2}) + \E\xi_t(L).
    \end{align*}
    Plugging in $t = \cO_t(\radt^{-(d+1)/2})$ leads to
    \begin{equation*}
        \Var\xi_t(L) = \cO_t (\radt^{(d+1)/2})  + \cO_t (\radt) + \cO_t(\radt^{d-1}) +  \cO_t (\radt^{(d-1)/2}) + \E\xi_t(L).
    \end{equation*}
    For $d\geq 3$, the $\cO_t(\radt)$-term is the largest one. For $d=2$, the largest term is $\cO_t(\sqrt{\radt})$.
\end{proof}

\section{Multivariate central limit theorem}
\label{sec:multivariate clt}
The central limit theorem concerns the thermodynamic and dense regimes, where $t\radt^d\geq c>0$ for all $t>0$. In this case, we can apply the results by \cite{Chimani2018} or the improvements on these results published only in the arXiv version \cite{chimani2020Arxiv} to scale the crossing number and stress.

We assume that $S(v_1,v_2;V)\leq s$ almost surely for all $(v_1,v_2)\in V^2_{\neq}$. Combining this with \cite[Theorems 13, 21 and 23]{chimani2020Arxiv} and their extensions to more general stress functionals in Appendix \ref{ap:variance}, we know that $$\Var \xi_t(L) = \cO_t(t^7 \radt^{4d+4}),\,\quad \Var (\stress(G,G_L)) = \cO_t(t^3), \text{ and}$$ $$\Cov(\xi_t(L),\stress(G,G_L)) = \cO_t(t^5 \radt^{2d+2}).$$
Note that $t^5\radt^{2d+2} = (t^3 t^7\radt^{4d+4})^{1/2}$.
We therefore scale the crossing number and the stress as follows:
\begin{equation*}
    F_t^{(1)} = \frac{\xi_t(L) - \E \xi_t(L)}{t^{7/2}\radt^{2d+2}},\quad
    F_t^{(2)} = \frac{\stress(G,G_L) - \E \stress(G,G_L)}{t^{3/2}}.
\end{equation*}

To prove the multivariate central limit theorem, we apply \cite[Theorem 1.1]{Schulte2019}, which requires first and second order difference operators. For a Poisson functional $F=f(\eta)$, where $\eta$ is a Poisson point process, these operators are defined as
\begin{align*}
    D_x F &= f(\eta + \delta_x) - f(\eta), \qquad
    D_{x,y}^2 F = f(\eta + \delta_x + \delta_y) - f(\eta + \delta_y) - f(\eta + \delta_x) + f(\eta).
\end{align*}
We need to find sufficiently sharp upper bounds for the following expressions from \cite[Theorem 1.1]{Schulte2019}:
\begin{align*}
    \gamma_1&\coloneqq
    \begin{multlined}[t]
        t^{3/2}\Bigg(\sum_{i,j=1}^2\int_{W^3} \l(\E (D_{x_1,x_3}^2 F_t^{(i)})^2 (D_{x_2,x_3}^2 F_t^{(i)})^2\r)^{1/2} \\
        \times \l(\E (D_{x_1}F_t^{(j)})^2 (D_{x_2}F_t^{(j)})^2\r)^{1/2}
        \dd x_1 \dd x_2 \dd x_3\Bigg)^{1/2},
    \end{multlined}\\
    \gamma_2 &\coloneqq 
    \begin{multlined}[t]
        t^{3/2}\Bigg(\sum_{i,j=1}^2\int_{W^3} \l(\E (D_{x_1,x_3}^2 F_t^{(i)})^2 (D_{x_2,x_3}^2 F_t^{(i)})^2\r)^{1/2} \\
        \times \l(\E (D_{x_1,x_3}^2 F_t^{(j)})^2 (D_{x_2,x_3}^2 F_t^{(j)})^2\r)^{1/2}
        \dd x_1 \dd x_2 \dd x_3\Bigg)^{1/2},
    \end{multlined}\\
    \gamma_3 &\coloneqq t\sum_{i=1}^2\int_W \E|D_x F_t^{(i)}|^3 \dd x.
\end{align*}
The Slivnyak-Mecke formula \eqref{eq:mecke} is particularly useful when calculating such expectations of difference operators.

We first calculate upper bounds for these expectations of the difference operators. For the stress functional $F_t^{(2)}$ we can derive sufficiently sharp upper bounds by only considering the maximal stress.
\begin{lem}[Bounds for stress]
    \label{lem:bounds stress}
    If $S(v_1,v_2;V)\leq s <\infty$ almost surely for all $v_1$ and $v_2$ then there exist constants $k_1$, $k_2$ and $k_3$ such that
    \begin{align}
        \label{eq:diff op 1 stress}
        \sup_{x_1,x_2\in W} \E (D_{x_1}F_t^{(2)})^2 (D_{x_2}F_t^{(2)})^2 &\leq k_1 t^{-2},\\
        \label{eq:diff op 1 stress 3rd power}
        \sup_{x\in W} \E|D_x F_t^{(2)}|^3 &\leq k_2 t^{-3/2},\\
        \label{eq:diff op 2 stress}
        \sup_{x_1,x_2,x_3 \in W}\E (D_{x_1,x_3}^2 F_t^{(2)})^2 (D_{x_2,x_3}^2 F_t^{(2)})^2 &\leq k_3 t^{-6},
    \end{align}
    for $t$ large enough.
\end{lem}
\begin{proof}
    For any $x\in W$, the first-order difference operator yields
    \begin{equation*}
        0 \leq D_{x} F_t^{(2)} = \sum_{v\in V} \frac{S(x,v;V)}{t^{3/2}} \leq |V|\frac{s}{{t^{3/2}}},
    \end{equation*}
    where $s$ is the upper bound for $S(x,v;V)$ and $|V|$ is the number of vertices.
    For a random variable $X\sim\Poi(t)$ we have $\E X^n = \cO(t^n)$ as $t$ goes to infinity.
    The bounds \eqref{eq:diff op 1 stress} and \eqref{eq:diff op 1 stress 3rd power} follow.

    The second-order difference operator only depends on the two points $x_1$ and $x_2$ added to the point process, so
    \begin{equation*}
        D^2_{x_1,x_2} F_t^{(2)} = \frac{S(x_1,x_2;V)}{t^{3/2}} \leq \frac{s}{t^{3/2}},
    \end{equation*}
    which results in the bound \eqref{eq:diff op 2 stress}.
\end{proof}
The first-order difference operators for the number of crossings are of the same order as those of the stress:
\begin{lem}[Bounds for number of crossings]
    \label{lem:bounds crossings}
    Let $t\radt^d\geq c>0$ for all $t$. Then there exist constants $k_4$ and $k_5$ such that
    \begin{align}
        \label{eq:diff op 1 cross}
        \sup_{x_1,x_2\in W} \E (D_{x_1}F_t^{(1)})^2 (D_{x_2}F_t^{(1)})^2 &\leq k_4 t^{-2},\\
        \label{eq:diff op 1 cross 3rd power}
        \sup_{x\in W} \E|D_x F_t^{(1)}|^3 &\leq k_5 t^{-3/2}.
    \end{align}
\end{lem}
\begin{proof}
    The first-order difference operator for the crossing number can be bounded from above as follows:
    \begin{align}
        D_x F_t^{(1)} &= \frac{1}{2}t^{-7/2} \radt^{-2d-2} \sum_{v,w,y\in V^3_{\neq}}
        \begin{multlined}[t]
            \indi{\|x-v\|\leq\radt, \|w-y\|\leq\radt}\\
            \cdot\indi{[x,v]|_L\cap [w,y]|_L \neq\emptyset}
        \end{multlined}
        \\
        &\leq \frac{1}{2}t^{-7/2} \radt^{-2d-2} \sum_{v,w,y\in V^3_{\neq}}
        \begin{multlined}[t]
            \indi{\|x-v\|\leq\radt, \|w-y\|\leq\radt}\\
            \cdot\indi{\|x|_L - w|_L\| \leq 2\radt}
        \end{multlined}
        \label{eq:bound 1st order crossings}
    \end{align}

    We have
    \begin{equation*}
        \E|D_x F_t^{(1)}|^3 \leq \frac{1}{8}t^{-21/2} \radt^{-6d-6}\E\sum_{\substack{v_1,w_1,y_1 \in V^3_{\neq} \\ v_2,w_2,y_2 \in V^3_{\neq} \\ v_3,w_3,y_3 \in V^3_{\neq}}}\prod_{i=1}^3
        \begin{multlined}[t]
            \indi{\|x-v_i\|\leq\radt, \|w_i-y_i\|\leq\radt}\\
            \cdot\indi{\|x|_L - w_i|_L\| \leq 2\radt}.
        \end{multlined}
    \end{equation*}
    We can apply the Slivnyak-Mecke formula to sums over distinct vertices in a Poisson point process. Since the nine vertices in this sum are not necessarily distinct, we have to distinguish between different cases: all vertices are distinct, eight vertices are distinct and two equal, etc. An application of the Slivnyak-Mecke formula for the expectation of the sum where all points are distinct yields
    \begin{equation*}
         \prod_{i=1}^3 \l( t^3\int_{W^3} \indi{\|x-v_i\|\leq\radt, \|w_i-y_i\|\leq\radt}
            \cdot\indi{\|x|_L - w_i|_L\| \leq 2\radt} \dd v_i \dd w_i \dd y_i \r).
    \end{equation*}
    The integral over points $v_i$ such that $\|x-v_i\|\leq\radt$ is of order $\radt^d$. The point $w_i$ must be placed in a cylinder with volume of order $\radt^2$. When $w_i$ is fixed, the point $y_i$ must be placed within a distance of $\radt$ from $w_i$ so that last integral is of order $\radt^d$ too. Each of the three integrals over $W^3$ is therefore of order $\cO_t(\radt^{2d + 2})$. Together with the prefactor the entire term is of order $\cO_t(t^{-21/2} \radt^{-6d-6} t^9 \radt^{6d+6}) = \cO_t(t^{-3/2})$.

    If the triples of points have some overlap, we sum over at most eight distinct points. An example of such a term is
    \begin{align*}
        \E&\sum_{\substack{v_1,w_1,y_1,v_2,y_2,v_3,w_3,y_3 \in V^8_{\neq} \\ w_1 = w_2}}\prod_{i=1}^3
        \indi{\|x-v_i\|\leq\radt, \|w_i-y_i\|\leq\radt} \cdot\indi{\|x|_L - w_i|_L\| \leq 2\radt}\\
        &=t^8 \int_{W^8}
        \begin{multlined}[t]
            \indi{\|x-v_1\|\leq\radt, \|w_1-y_1\|\leq\radt} \cdot\indi{\|x|_L - w_1|_L\| \leq 2\radt}^2\\
            \cdot \indi{\|x-v_2\|\leq\radt, \|w_1-y_2\|\leq\radt}
            \cdot \indi{\|x-v_3\|\leq\radt, \|w_3-y_3\|\leq\radt}\\ \cdot\indi{\|x|_L - w_3|_L\| \leq 2\radt}
            \dd v_1 \dd w_1 \dd y_1 \dd v_2 \dd y_2 \dd v_3 \dd w_3 \dd y_3.
        \end{multlined}
    \end{align*}
    The order of this integral can be determined the same way as before with the nine distinct vertices and leads to a contribution of order $\cO_t(t^{-{5/2}}\radt^{-2})$. In general, one can see that restricting two vertices to be equal results in a loss of an integral of order $\cO_t(t\radt^2)$ or $\cO_t(t\radt^d)$. It is then clear that the restriction that two or more vertices are the same does not increase the order of this expectation. These other terms are therefore also at most $\cO_t(t^{-3/2})$. We can conclude that $\E|D_x F_t^{(1)}|^3 = \cO_t(t^{-3/2})$.
    
    Similar calculations with tuples of twelve points show that 
    \begin{align*}
        \sup_{x_1,x_2\in W} \E (D_{x_1}F_t^{(1)})^2 (D_{x_2}F_t^{(1)})^2 &\leq \sup_{x_1,x_2\in W} \E[ (D_{x_1}F_t^{(1)})^4 + (D_{x_2}F_t^{(1)})^4]\\
        &= 2\sup_{x\in W} \E (D_{x}F_t^{(1)})^4 = \cO_t(t^{-2})
    \end{align*}
    which is what we needed to show for \eqref{eq:diff op 1 cross}.
\end{proof}

\begin{lem}
    Under the conditions stated in Theorem \ref{thm:multivariate CLT} the parameters $\gamma_1$, $\gamma_2$ and $\gamma_3$ are of the following order:
    \begin{equation*}
        \gamma_1 = \cO_t(t^{-1/2}),\
        \gamma_2 = \cO_t(t^{-1/2}),\ \text{and }
        \gamma_3 = \cO_t(t^{-1/2}).
    \end{equation*}
\end{lem}
\begin{proof}
    The order of $\gamma_3$ follows immediately from Lemma \ref{lem:bounds stress} and \ref{lem:bounds crossings}. For $\gamma_1$, we obtain by \eqref{eq:diff op 1 stress} and \eqref{eq:diff op 1 cross}
    \begin{equation*}
        \gamma_1 \leq (\sup\{k_1, k_4\})^{1/4} t\Bigg(\sum_{i,j=1}^2\int_{W^3} \l(\E (D_{x_1,x_3}^2 F_t^{(i)})^2 (D_{x_2,x_3}^2 F_t^{(i)})^2\r)^{1/2}
        \dd x_1 \dd x_2 \dd x_3\Bigg)^{1/2},
    \end{equation*}
    which implies
    \begin{multline*}
        \gamma_1\leq 2
        (\sup\{k_1, k_4\})^{1/4} t\\
        \bigg(k_3^{1/2}t^{-3} + \int_{W^3} \l(\E (D_{x_1,x_3}^2 F_t^{(1)})^2 (D_{x_2,x_3}^2 F_t^{(1)})^2\r)^{1/2}
        \dd x_1 \dd x_2 \dd x_3\bigg)^{1/2}
    \end{multline*}
    by \eqref{eq:diff op 2 stress} and $\lambda_d(W)=1$.
    
    Note that $D_{x_1,x_3}^2 F_t^{(1)}$ counts the number of crossings with both vertices $x_1$ and $x_3$ involved which either means that $x_1$ and $x_3$ form a new edge and intersect an existing one or that $x_1$ and $x_3$ are endpoints of distinct edges which cross in their projection:
    \begin{align*}
        D_{x_1,x_3}^2 &\begin{multlined}[t]
        F_t^{(1)}=\\
                \frac{1}{2t^{7/2}\radt^{2d+2}}\!
                \sum_{v_1,v_3\in V^2_{\neq}}\!\indi{\|x_1-x_3\|\leq\radt,\|v_1-v_3\|\leq\radt}\indi{[x_1,x_3]|_L\cap [v_1,v_3]|_L \neq\emptyset}\\
                + \indi{\|x_1-v_1\|\leq\radt,\|x_3-v_3\|\leq\radt}\indi{[x_1,v_1]|_L\cap [x_3,v_3]|_L \neq\emptyset}.
        \end{multlined}\\
        &\leq
            \begin{multlined}[t]
                \frac{1}{2t^{7/2}\radt^{2d+2}}\!\sum_{v_1,v_3\in V^2_{\neq}}\!\indi{\|x_1-x_3\|\leq\radt,\|v_1-v_3\|\leq\radt} \indi{\|x_1|_L - v_1|_L\|\leq 2\radt}\\
                + \indi{\|x_1-v_1\|\leq\radt,\|x_3-v_3\|\leq\radt}\indi{\|x_1|_L - x_3|_L\|\leq 2\radt}.
        \end{multlined}
        \label{eq:bound 2nd order crossings}
    \end{align*}
    
    Multiplying four such expressions to bound $\gamma_1$ and $\gamma_2$ leads to a sum of products of indicator functions in a way similar to the proof of Lemma \ref{lem:bounds crossings}, where this time the sum is over four pairs of points from $V$. These pairs of points may have overlap, and the products of indicators can consist of combinations of the two different products we see in the equation above. Calculating $\gamma_1$ and $\gamma_2$ therefore requires many applications of the Slivnyak-Mecke formula to all various cases. However, one again quickly sees that the highest order terms are the ones where the eight-tuples over which we sum consist of distinct vertices.

    We distinguish between the two extreme cases: either $\|x_1-x_2\|\leq r_t$ and $\|x_1-x_3\|\leq r_t$ or $\|x_1-x_2\|> r_t$ and $\|x_1-x_3\|> r_t$. 
    If $x_1$ and $x_3$ form an edge, each integral
        $$\int_{W^2}\indi{\|x_1-x_3\|\leq\radt,\|v_i-v_j\|\leq\radt} \indi{\|x_1|_L - v_i|_L\|\leq 2\radt}\dd v_i\, \dd v_j$$
    resulting from an application of the Slivnyak-Mecke formula contributes a factor of order $(t \radt^d)(t\radt^2) = t^2\radt^{d+2}$. The factors
        $$\int_{W^2}\indi{\|x_1-v_i\|\leq\radt,\|x_3-v_j\|\leq\radt}\indi{\|x_1|_L - x_3|_L\|\leq 2\radt}\dd v_i\, \dd v_j$$
    each contribute a term of order $(t\radt^d)(t\radt^d) = t^2\radt^{2d}$ whenever $\|x_1|_L - x_3|_L\|\leq 2\radt$.

    If $\|x_1-x_2\|\leq r_t$ and $\|x_1-x_3\|\leq r_t$, the points $x_2$ and $x_3$ lie in a volume of order $r_t^d$ so the integral contributes a factor of order $r_t^{2d}$. The expectation contributes a factor of order
    \begin{equation*}
        \l(\E (D_{x_1,x_3}^2 F_t^{(1)})^2 (D_{x_2,x_3}^2 F_t^{(1)})^2\r)^{1/2} = \cO_t\Big(\frac{(t^2\radt^{d+2})^2}{t^7\radt^{4d+4}}\Big) = \cO_t(t^{-3}\radt^{-2d}).
    \end{equation*}
    If $\|x_1-x_2\|> r_t$ and $\|x_1-x_3\|> r_t$, the points $x_2$ and $x_3$ must each lie in a volume of order $\radt^2$ in order for the integrand to be non-zero. The contribution of the expectation is a factor of order
    \begin{equation*}
        \l(\E (D_{x_1,x_3}^2 F_t^{(1)})^2 (D_{x_2,x_3}^2 F_t^{(1)})^2\r)^{1/2} = \cO_t\Big(\frac{(t^2\radt^{2d})^2}{t^7\radt^{4d+4}}\Big) = \cO_t(t^{-3}\radt^{-4}).
    \end{equation*}

    Combining the above calculations, we obtain 
    \begin{equation*}
        \int_{W^3} \l(\E (D_{x_1,x_3}^2 F_t^{(1)})^2 (D_{x_2,x_3}^2 F_t^{(1)})^2\r)^{1/2}
        \dd x_1 \dd x_2 \dd x_3=\cO_t(t^{-3}).
    \end{equation*}
    It follows that $\gamma_1 = \cO_t(t^{-1/2})$.

    The order of $\gamma_2$ is derived similarly. The largest term, which determines the order, comes from the number of crossings:
    \begin{equation*}
        \int_{W^3} \E (D_{x_1,x_3}^2 F_t^{(1)})^2 (D_{x_2,x_3}^2 F_t^{(1)})^2
        \dd x_1 \dd x_2 \dd x_3 = \cO_t(t^{-4}),
    \end{equation*}
    \begin{equation*}
        \int_{W^3} \E (D_{x_1,x_3}^2 F_t^{(2)})^2 (D_{x_2,x_3}^2 F_t^{(2)})^2
        \dd x_1 \dd x_2 \dd x_3 = \cO_t(t^{-6}),
    \end{equation*}
    \begin{multline*}
        \int_{W^3} \l(\E (D_{x_1,x_3}^2 F_t^{(1)})^2 (D_{x_2,x_3}^2 F_t^{(1)})^2\r)^{1/2}\\
        \times\l(\E (D_{x_1,x_3}^2 F_t^{(2)})^2 (D_{x_2,x_3}^2 F_t^{(2)})^2\r)^{1/2}
        \dd x_1 \dd x_2 \dd x_3 = \cO_t(t^{-5}).
    \end{multline*}
    Hence $\gamma_2 = \cO_t(t^{-1/2})$
\end{proof}

The obtained parameters $\gamma_1$, $\gamma_2$ and $\gamma_3$ can now be combined with \cite[Theorem 1.1]{Schulte2019} to prove the following proposition and Theorem \ref{thm:multivariate CLT}.
\begin{prop}
    Suppose that the covariance matrix $\Sigma_t$ of $F_t$ converges to $\Sigma = (\sigma_{ij})_{i,j\in\{1,2\}}$. Then
    \begin{equation*}
        \dd_3(F_t, Z_\Sigma) \leq \sum_{i,j = 1}^2 |\sigma_{ij} - \Cov(F_t^{(i)},F_t^{(j)})| +  \cO_t(t^{-1/2}),
    \end{equation*}
    where $Z_{\Sigma}\sim\cN(0,\Sigma)$.
    \label{prop:clt}
\end{prop}
\begin{proof}
    We have
    \begin{equation*}
        2\gamma_1 + \gamma_2 + \gamma_3 = \cO_t(t^{-1/2}).
    \end{equation*}
    The result then immediately follows from an application of \cite[Theorem 1.1]{Schulte2019}.
\end{proof}

\begin{proof}[Proof of Theorem \ref{thm:multivariate CLT}]
    Similarly to the proof of Proposition \ref{prop:clt}, we have
    \begin{equation*}
        \dd_3(F_t, Z_{\Sigma_t}) \leq \sum_{i,j = 1}^2 |\Cov(F_t^{(i)},F_t^{(j)}) - \Cov(F_t^{(i)},F_t^{(j)})| +  \cO_t(t^{-1/2}) = \cO_t(t^{-1/2}).
    \end{equation*}
    
    The convergence in distribution of $F_t$ to $\cN(0,\Sigma)$ follows from Proposition \ref{prop:clt}. Since $\Cov(F_t^{(i)},F_t^{(j)}) \to \sigma_{ij}$ for $i,j\in\{1,2\}$, we know that the upper bound on $\dd_3(F_t, Z_\Sigma)$ goes to zero, and we obtain the desired result.
\end{proof}

\appendix
\section{Variance of stress and covariance}
\label{ap:variance}
The results from \cite{Chimani2018} on the stress assume that the stress is a $U$-statistic, i.e. that $S(v_1,v_2;V)$ is independent of $V\setminus\{v_1,v_2\}$. The goal of this appendix is to extend the results from \cite{Chimani2018} for the stress to a broader class of stress functionals, where the metrics $\dd_0$ and $\dd_L$ can also depend on the vertices of the graph. These derivations follow essentially the same steps as the proofs in \cite{Chimani2018}.

Applying the Slivnyak-Mecke formula to the expectation of the stress, we obtain
\begin{align*}
    \E \stress(G,G_L) &= \E \frac{1}{2}\sum_{(v_1,v_2)\in V_{\neq}^2} S(v_1,v_2; V)\\
    &= \frac{1}{2} t^2 \int_{W^2} \E \l[S(v_1,v_2; V\cup\{v_1,v_2\})\r] \dd v_1 \dd v_2.
\end{align*}
For the calculation of the variance, we follow the proof of \cite[Theorem 21]{chimani2020Arxiv} and again apply the Slivnyak-Mecke formula. Then,
\begin{align*}
    \Var \stress(G,G_L) &= \E \sum_{(v,v_2,w_2)\in V^3_{\neq}} S(v,v_2;V)S(v,w_2;V) + \frac{1}{2}\E \sum_{(v_1,v_2)\in V^2_{\neq}} S(v_1,v_2;V)^2\\
    &= \begin{multlined}[t]
        t^3\int_{W^3} \E S(v,v_2;V\cup\{v,v_2\})S(v,w_2;V\cup\{v,w_2\})\dd v_2 \dd w_2 \dd v \\
        + \frac{1}{2}t^2 \int_{W^2} \E S(v_1,v_2;V\cup\{v_1,v_2\})^2 \dd v_1 \dd v_2
    \end{multlined}
    \\
    &= t^3\int_{W^3} \E S(v,v_2;V\cup\{v,v_2\})S(v,w_2;V\cup\{v,w_2\})\dd v_2 \dd w_2 \dd v + \cO(t^2).
\end{align*}
This expression is independent of whether or not we project the graph and is therefore also valid for $d=2$.

For the covariance, we follow the proof of \cite[Theorem 23]{chimani2020Arxiv}. We then obtain
\begin{multline*}
    \Cov(\xi_t(L), \stress(G,G_L)) = \frac{1}{2} c_d t^5 \radt^{2d+2}\int_W \lambda_{d-2}((v+L^{\perp})\cap W) \\
    \cdot\int_{W} \E\l[S(v,w; V\cup \{v,w\})\r] \dd w \dd v\,(1+o(1)).
\end{multline*}
When $d=2$, we obtain following similar calculations
\begin{equation*}
    \Cov(\xi_t(L), \stress(G,G_L)) = \frac{1}{2} c_2 t^5 \radt^{6}\int_{W^2} \E\l[S(v,w; V\cup \{v,w\})\r] \dd w \dd v\,(1+o(1)).
\end{equation*}

The vertex set $V$ depends on the intensity $t$ of the point process. The metrics $\dd_0$ and $\dd_L$ and the weight function appearing in the stress functional can be chosen to depend on the intensity $t$ too. In order for the covariance matrix in Theorem \ref{thm:multivariate CLT} to converge, the integrals in the variance and covariance above must converge.

\section{Central limit theorem rate of convergence for $W = [0,1]^d$}
\label{sec:cube}

To determine how fast $F_t$ converges to the normal distribution, we need a rate of convergence of the variance and the covariance of the crossing number and the stress. For some choices of $W$, these rates can be derived by making some adjustments to the derivations of the variance and covariance in \cite{chimani2020Arxiv}. The easiest choice of $W$ is the $d$-dimensional cube because of its constant height with respect to $\R^2\times\{0\}^{d-2}$. 
\begin{prop}
    Let $W=[0,1]^d$ and $L = \R^2 \times \{0\}^{d-2}$ and let the stress be independent of the vertex set $V$ of $G$. Define $\radt$ such that $t\radt^d=c$ for all $t>0$ and let $\Sigma$ be defined as in Theorem \ref{thm:multivariate CLT}.  Then for $Z_{\Sigma}\sim\cN(0,\Sigma)$,
    \begin{equation*}
        \dd_3(F_t,Z_{\Sigma})) = \cO(\radt)
    \end{equation*}
    as $t$ goes to infinity.
    \label{prop:clt cube}
\end{prop}
The improved approximations of the variance and covariance needed for the derivation of this result are given in the following two lemmas:
\begin{lem}
    Consider the model described in Proposition \ref{prop:clt cube}. The variance of the crossing number is
    \begin{equation}
        \Var \xi_t(L) = \frac{1}{8} c^4 t^3 \radt^{4}(2c_d^2 + c_d' c^{-1})(1+\cO_t(\radt)),
        \label{eq:variance improved}
    \end{equation}
    where $c_d' = \pi \kappa_{d-2}^3 \bfB(3,d/2)^2\bfB(4,d/2)$ and $c_d$ is defined as in Lemma \ref{lem:intensity measure}.
    \label{lem:variance cube}
\end{lem}

\begin{lem}
    Consider the model described in Proposition \ref{prop:clt cube}. For the covariance between the crossing number and the stress it holds that
    \begin{equation}
        \Cov(\xi_t(W_L),\stress(G,G_L)) = \frac{1}{2} c_d t^3\radt^{2}\int_W S_{W,L}^{(1)}(v)\dd v (1+ \cO_t(\radt)),
        \label{eq:covariance improved}
    \end{equation}
    where 
    \begin{equation*}
        S_{W,L}^{(1)}(v) \coloneqq \int_{W-v} w(0,v_2) (\dd_0(0,v_2) - \dd_L(0,v_2))^2 \dd v_2.
    \end{equation*}
    \label{lem:covariance cube}
\end{lem}
We first prove Proposition \ref{prop:clt cube}, after which we will prove the lemmas too.
\begin{proof}[Proof of Proposition \ref{prop:clt cube}]
    From \eqref{eq:variance improved}, we can calculate
    \begin{equation*}
        \Var F_t^{(1)} = \frac{\Var \xi_t(W_L)}{t^7 \radt^{4d+4}} = \frac{1}{8} (2c_d^2 + c_d' c^{-1}) + \cO_t(\radt)
    \end{equation*}
    The variance of $F_t^{(2)}$ follows from \cite[Theorem 6]{Chimani2018}:
    \begin{equation*}
        \Var F_t^{(2)} = \frac{\Var \stress(G,G_L)}{t^3} = \int_W S_{W,L}^{(1)}(v)^2\dd v + \cO_t(t^{-1}).
    \end{equation*}
    The rate of convergence for the covariance $\Cov(F_t^{(1)}, F_t^{(2)})$ is given by
    \begin{equation*}
        \Cov( F_t^{(1)}, F_t^{(2)}) = \frac{\Cov(\xi_t(W_L),\stress(G,G_L))}{t^5\radt^{2d+2}} = \frac{1}{2} c_d \int_W S_{W,L}^{(1)}(v)\dd v  + \cO_t(\radt)
    \end{equation*}
    It follows from Proposition \ref{prop:clt} that
    \begin{equation*}
        \dd_3(F_t,Z_{\Sigma}) \leq  \cO(\radt) + \cO(t^{-1}) + \cO(t^{-1/2}) = \cO(\radt)
    \end{equation*}
    as $t$ tends to infinity.
\end{proof}

The proofs of Lemma \ref{lem:variance cube} and Lemma \ref{lem:covariance cube} mostly follow the same steps as in the derivations of these quantities in \cite{chimani2020Arxiv}. We first summarize the bounds they derived before using them in our proofs, sticking to the notation from \cite{chimani2020Arxiv} as much as possible.

\paragraph{Bounds from \cite{chimani2020Arxiv}}
For a convex body $K$ and a $2$-dimensional plane in $\R^d$, the following integral is defined in \cite[Lemma 8]{chimani2020Arxiv}:
\begin{equation*}
    J_L^{(1)}(K) \coloneqq \int_{(\radt B_d) \times K \times (\radt B_d)} \indi{[0,x]|_L \cap (y+[0,z])|_L\neq \emptyset}\dd z \dd y \dd x.
\end{equation*}
Defining $c_d$ as in Lemma \ref{lem:intensity measure}, the proof of \cite[Lemma 8]{chimani2020Arxiv} contains the following inequalities:
\begin{multline*}
    c_d\radt^{2d+2}\indi{2\radt B_2\subset K |_L} \inf_{u\in \radt B_d} \lambda_{d-2}((u+L^{\perp})\cap K) \leq J_L^{(1)}(K) \\
    \leq c_d\radt^{2d+2}\sup_{u\in \radt B_d} \lambda_{d-2}((u+L^{\perp})\cap K).
\end{multline*}
The following integral is also defined in \cite{chimani2020Arxiv} for points $v\in W$:
\begin{equation*}
    I_{W,L}^{(1)}(v) \coloneqq \int_{W_3} \indi{[v,v_2]_L\cap[v_3,v_4]_L \neq \emptyset,\, \|v-v_2\|\leq \radt,\, \|v_3-v_4\|\leq \radt}\dd v_2 \dd v_3 \dd v_4.
\end{equation*}
By the proof of \cite[Proposition 9]{chimani2020Arxiv}
\begin{equation*}
    J_L^{(1)}(W_{-\radt}-v) \leq I_{W,L}^{(1)}(v) \leq J_L^{(1)}(W-v).
\end{equation*}
The following integral is an analogue of $J_L^{(1)}(K)$ needed to calculate the variance of the crossing number:
\begin{multline*}
    J_L^{(2)}(K) \coloneqq \int_{(\radt B_d) \times K^2 \times (\radt B_d)^2} \indi{[0,x]|_L \cap (y_1+[0,z_1])|_L\neq \emptyset} \\
    \cdot\indi{[0,x]|_L \cap (y_2+[0,z_2])|_L\neq \emptyset} \dd z_1 \dd z_2 \dd y_1 \dd y_2 \dd x
\end{multline*}
Defining $c_d' = \pi \kappa_{d-2}^3 \bfB(3,d/2)^2\bfB(4,d/2)$, it is derived in the proof of \cite[Lemma 10]{chimani2020Arxiv} that
\begin{multline*}
    c_d'\radt^{3d+4}\indi{2\radt B_2 \subset K|_L}\big(\inf_{u\in \radt B_d} \lambda_{d-2}((u+L^{\perp})\cap K)\big)^2 \leq J_L^{(2)}(K)\\
    \leq c_d'\radt^{3d+4}\big(\sup_{u\in \radt B_d} \lambda_{d-2}((u+L^{\perp})\cap K)\big)^2.
\end{multline*}
Letting $v\in W$ again, the following integral is defined in \cite[Proposition 11]{chimani2020Arxiv}:
\begin{multline*}
    I_{W,L}^{(2)}(v) = \int_{W_5}                   
         \indi{[v,v_2]_L\cap[v_3,v_4]_L \neq \emptyset,\, \|v-v_2\|\leq \radt,\, \|v_3-v_4\|\leq \radt}\\
         \cdot\indi{[v,v_2]_L\cap[w_3,w_4]_L \neq \emptyset,\, \|w_3-w_4\|\leq \radt}\dd v_2 \dd v_3 \dd v_4 \dd w_3 \dd w_4.
\end{multline*}
In the proof of the same proposition, it is derived that
\begin{equation*}
    J_L^{(2)}(W_{-\radt}-v) \leq I_{W,L}^{(2)}(v) \leq J_L^{(2)}(W-v).
\end{equation*}

\paragraph{Proofs of the lemmas}
\begin{proof}[Proof of Lemma \ref{lem:variance cube}]
    We closely follow the proof of \cite[Theorem 13]{chimani2020Arxiv}. By this proof, the variance is given by
    \begin{multline}
        \Var \xi_t(L) = \frac{1}{4} t^7 \int_W I_{W,L}^{(1)}(v)^2 \dd v + \frac{1}{8} t^6 \int_W I_{W,L}^{(2)}(v) \dd v \\
        + \cO_t(t^6\radt^{4d+2}) + \cO_t(t^5\radt^{3d+2}) + \cO_t(t^4\radt^{2d+2}).
        \label{eq:variance from proof}
    \end{multline}

    Let $v\in W_{-3\radt}$. Then we use the bounds prepared by \cite{chimani2020Arxiv} to bound $I_{W,L}^{(1)}(v)$ from below:
    \begin{align*}
        I_{W,L}^{(1)}(v) &\geq J_L^{(1)}(W_{-\radt}-v)\\
        &\geq c_d \radt^{2d+2} \inf_{u\in \radt B_d} \lambda_{d-2}((u+L^{\perp})\cap (W_{-\radt}-v))\\
        &= c_d \radt^{2d+2} (1-2\radt)^{d-2},
    \end{align*}
    where $(1-2\radt)^{d-2}$ is the volume of a $(d-2)$-dimensional cube with edge length $1-2\radt$. For $v\in W\setminus  W_{-3\radt}$, a trivial lower bound for $I_{W,L}^{(1)}(v)$ is zero.
    A similar derivation for the upper bound yields for all $v\in W$
    \begin{align*}
        I_{W,L}^{(1)}(v) &\leq J_L^{(1)}(W-v)\\
        &\leq c_d\radt^{2d+2}\sup_{u\in \radt B_d} \lambda_{d-2}((u+L^{\perp})\cap (W-v))\\
        &= c_d\radt^{2d+2}.
    \end{align*}
    It follows that
    \begin{align*}
        \frac{1}{4} t^7 \int_W I_{W,L}^{(1)}(v)^2 \dd v &= \frac{1}{4} t^7 \l(\int_{W_{-3\radt}} I_{W,L}^{(1)}(v)^2 \dd v + \int_{W\setminus W_{-3\radt}} I_{W,L}^{(1)}(v)^2 \dd v\r)\\
        &= \frac{1}{4} t^7\l(c_d^2\radt^{4d+4}(1+\cO_t(\radt)) + c_d^2\radt^{4d+4}\cdot \cO_t(\radt)\r)\\
        &= \frac{1}{4} t^7\radt^{4d+4}c_d^2(1+\cO_t(\radt)).
    \end{align*}

    For the second integral in the variance, we can also use the bounds from \cite{chimani2020Arxiv}. Let $v\in W_{-3\radt}$, then
    \begin{align*}
        I_{W,L}^{(2)}(v) &\geq J_L^{(2)}(W_{-\radt}-v)\\
        &\geq c_d'\radt^{3d+4}\big(\inf_{u\in \radt B_d} \lambda_{d-2}((u+L^{\perp})\cap (W_{-\radt}-v))\big)^2)\\
        &= c_d' \radt^{3d+4} (1-2\radt)^{2(d-2)}.
    \end{align*}
    Zero is again a trivial lower bound for the case where $v\in W\setminus W_{-3\radt}$.
    For the upper bound, we again let $v\in W$ and derive
    \begin{align*}
        I_{W,L}^{(2)}(v) &\leq J_L^{(2)}(W-v)\\
        &\leq c_d'\radt^{3d+4}\big(\sup_{u\in \radt B_d} \lambda_{d-2}((u+L^{\perp})\cap (W-v))\big)^2)\\
        &= c_d' \radt^{3d+4}.
    \end{align*}
    Then,
    \begin{align*}
        \frac{1}{8} t^6 \int_W I_{W,L}^{(2)}(v) \dd v &= \frac{1}{8} t^6 \l(\int_{W_{-3\radt}} I_{W,L}^{(2)}(v) \dd v + \int_{W\setminus W_{-3\radt}} I_{W,L}^{(2)}(v) \dd v\r)\\
        &= \frac{1}{8} t^6 \l(c_d' \radt^{3d+4}(1+\cO_t(\radt) + c_d' \radt^{3d+4} \cdot (\radt)\r)\\
        &= \frac{1}{8} t^6 c_d' \radt^{3d+4}(1+\cO_t(\radt)).
    \end{align*}

    We can plug the expressions we obtained for the integrals into \eqref{eq:variance from proof} to see that
    \begin{equation*}
        \Var \xi_t(L) = \frac{1}{8} t^7 \radt^{4d+4}(2 c_d^2 + c_d' t^{-1}\radt^{-d})(1+\cO_t(\radt)).
    \end{equation*}
    Setting $t\radt^d = c$ leads to \eqref{eq:variance improved}.
\end{proof}

\begin{proof}[Proof of Lemma \ref{lem:covariance cube}]
    By the proof of \cite[Theorem 23]{chimani2020Arxiv}, the covariance is given by
    \begin{equation*}
        \Cov(\xi_t(W_L),\stress(G,G_L)) = \frac{1}{2} t^5\int_W I_{W,L}^{(1)}(v) S_{W,L}^{(1)}(v)\dd v + \cO_t(t^4\radt^{2d+2}).
    \end{equation*}
    Using our bounds on $I_{W,L}^{(1)}$, we can write
    \begin{align*}
        \Cov(\xi_t(W_L),\stress(G,G_L)) &\geq 
        \begin{multlined}[t]
            \frac{1}{2}c_d t^5 \radt^{2d+2} \int_{W} \big(\inf_{u\in \radt B_d} \lambda_{d-2}((u+L^{\perp})\cap (W_{-\radt}-v))\big)\\
            \cdot S_{W,L}(v)\dd v
        \end{multlined}
        \\
        &\geq
        \frac{1}{2}c_d t^5 \radt^{2d+2}(1-2\radt)^{d-2} \int_{W_{-\radt}}S_{W,L}(v)\dd v
    \end{align*}
    and
    \begin{align*}
        \Cov(\xi_t(W_L),\stress(G,G_L)) &\leq 
        \begin{multlined}[t]
            \frac{1}{2}c_d t^5 \radt^{2d+2} \int_{W} \big(\sup_{u\in 2\radt B_d} \lambda_{d-2}((u+L^{\perp})\cap (W-v))\big)\\
            \cdot S_{W,L}(v)\dd v
        \end{multlined}\\
        &=
            \frac{1}{2}c_d t^5 \radt^{2d+2} \int_{W}S_{W,L}(v)\dd v.
    \end{align*}
    As before in the proof of Lemma \ref{lem:variance cube}, we can look at the difference between the upper and lower bound and use that $S_{W,L}(v) \leq s$ for all $v\in W$ to obtain \eqref{eq:covariance improved}.
\end{proof}

\begin{rem}
    From this proof, we can also obtain a uni-variate central limit theorem for the crossings: $\dd_3(F_t^{(1)}, Z_{\sigma_{11}}) = \cO_t(\radt)$, where $Z_{\sigma{11}}\sim \cN(0,\sigma_{11})$.
\end{rem}

\paragraph{Remarks on the rate of convergence}
The proofs in Appendix \ref{sec:cube} show how to calculate the entries of the covariance matrix when $W$ is a cube. It might be possible to derive an order of convergence for a broader class of windows in a similar way to the proof of Lemma \ref{lem:uniform error bound}, but this is significantly more complicated.

The rate of convergence for a cube $W = [0,1]^d$ is of order $\radt$ because of boundary effects for the crossings. Because of these boundary effects, we expect the variance of the crossings and the covariance to be the dominating terms in the upper bound of $\dd_3(F_t, Z_\Sigma)$ for other choices of $W$ too.

\vspace{2em}
\section*{Acknowledgements}
We wish to thank the referees for their thorough reading and helpful comments,
which greatly improved the arguments presented in this paper.

\section*{Funding information}
Funded by the Deutsche Forschungsgemeinschaft (DFG, German Research Foundation) – Project-ID 531542011.

\addcontentsline{toc}{section}{References}
\bibliographystyle{abbrv}
\bibliography{ref}

\end{document}